\documentclass[11pt,reqno]{amsart}

\usepackage{hyperref}
\usepackage{amssymb,amsfonts,url}
\usepackage[all,arc]{xy}
\usepackage{enumerate}
\usepackage{mathrsfs}
\usepackage[utf8]{inputenc}
\usepackage[margin=1in]{geometry}
\usepackage{caption}
\usepackage[pagewise]{lineno}
\usepackage{graphicx}
\usepackage{amsaddr}

\newtheorem{thm}{Theorem}[section]

\newtheorem{lem}[thm]{Lemma}

\newtheorem{prob}[thm]{Problem}

\theoremstyle{definition}

\theoremstyle{theorem}
\newtheorem{rem}[thm]{Remark}

\makeatletter
\let\c@equation\c@thm
\makeatother
\numberwithin{equation}{section}

\title[Blow-up of solutions of NLS]{Blow-up of solutions of nonlinear Schrödinger equations with oscillating nonlinearities}

\author{Türker Özsarı}
\email{turkerozsari@iyte.edu.tr}
\address{Department of Mathematics, Izmir Institute of Technology, Urla, Izmir, 35430 TURKEY}

\date{}

\begin{document}

\begin{abstract}
The finite time blow-up of solutions for 1-D NLS with oscillating nonlinearities is shown in two domains: (1) the whole real line where the nonlinear source is acting in the interior of the domain and (2) the right half-line where the nonlinear source is placed at the boundary point.  The distinctive feature of this work is that the initial energy is allowed to be non-negative and the momentum is allowed to be infinite in contrast to the previous literature on the blow-up of solutions with time dependent nonlinearities.   The common finite momentum assumption is removed by using a compactly supported or rapidly decaying weight function in virial identities - an idea borrowed from \cite{OT}.  At the end of the paper, a numerical example satisfying the theory is provided.
\end{abstract}

\maketitle


\section{Introduction and Main Results}

In this article we study two problems.  The first problem concerns the nonlinear Schrödinger equation posed on $\mathbb{R}$ with a critical focusing type source term whose coefficient is oscillating.

\begin{equation}\label{Model1a}
  i\frac{\partial u}{\partial t} = -\frac{\partial^2u}{\partial x^2}-A_\Omega(t)|u|^4u, t>t_0, x\in\mathbb{R},
\end{equation}
\begin{equation}\label{Model1b}
  u(t_0,x)=u_0(x), x\in\mathbb{R},
\end{equation}
where $t_0\in \mathbb{R}$, $A_\Omega(t)= a(2\Omega t)$ for some real valued periodic smooth function $a(t)$, and $\Omega>0$ is a fixed constant.  $a(t)=\cos^2(\frac{t}{2})$ is a common example.

The second problem examines the Schrödinger equation on the right half-line with a critical or supercritical nonlinear focusing type source term, which has a time dependent coefficient, acting at the only boundary point.

\begin{equation}\label{Model2a}
  i\frac{\partial u}{\partial t} = -\frac{\partial^2u}{\partial x^2}, t>t_0, x>0,
\end{equation}
\begin{equation}\label{Model2b}
  \lim_{x\rightarrow 0^+}u_x(t,x)=-A_\Omega(t)|u(t,0)|^ru(t,0), t>t_0,
\end{equation}
\begin{equation}\label{Model2c}
  u(t_0,x)=u_0(x), x>0,
\end{equation} where $t_0$, $\Omega$, $A_\Omega$, and $a$ are as in \eqref{Model1a}-\eqref{Model1b}, and $r\ge 2$.

Our aim for both problems is to point to the fact that for appropriate initial data $u_0$, it is always possible to obtain solutions which blow-up at the energy level in finite time, even if the initial momentum is infinite and energy is non-negative.

\subsection{Physical motivation}
The nonlinear Schrödinger equation (NLS) is a classical field equation.  It became popular when its one-dimensional version was shown to be integrable in \cite{ZS}.  Although  it has many applications in physics, evolution of a quantum state is not described by NLS, which contrasts with the use of linear Schrödinger equation.   Applications of NLS include transmission of light in nonlinear optical fibers and planar wavequides, small-amplitude gravity waves on the surface of deep inviscid water, and Langmuir waves in hot plasmas \cite{Sulem}, \cite{MB}.  NLS also appears as a universal equation governing the evolution of slowly varying packets of quasi-monochromatic waves in weakly nonlinear dispersive media \cite{Sulem}, \cite{MB}.  Some other interesting applications of NLS include Bose-Einstein condensates \cite{PS}, Davydov's alpha-helix solitons \cite{BR}, and plane-diffracted wave beams in the focusing regions of the ionosphere \cite{AVG}.

\subsubsection*{Whole line model} The NLS with fifth order power type nonlinearity, in particular the one dimensional model given in \eqref{Model1a}-\eqref{Model1b} is referred to as the quintic Schrödinger equation. The quintic NLS (with say $A_\Omega$ being a constant) \begin{equation}\label{Quintic}
              iu_t+u_{xx}-A_\Omega |u|^4u=0
            \end{equation}  is an important model from both mathematical and physical perspectives although it has never been as popular as the NLS with cubic nonlinearity.  One can show that the quintic NLS \eqref{Quintic} is equivalent to a Hamiltonian structure (see e.g., \cite{Sulem}) given by $$\displaystyle\frac{\partial u}{\partial t}=i\frac{\delta H[u]}{\delta \bar{u}},$$ where the Hamiltonian is defined by $$H[u]=\|u_x\|_{L^2}^2+\frac{A_\Omega}{3}\int_{R^n} |u|^6dx.$$  This type of structure naturally arises when $u$ represents a slow envelope of a small amplitude wave in a region which carries the properties of weak dispersion and weak nonlinearity depending on the intensity of the wave \cite{alf2007}.

The quintic Schrödinger equation was recommended as a model for some concrete physical applications.  For instance, \cite{Kol2000} recommended this model for describing the evolution of one-dimensional gas of impenetrable bosons.  Reliability of this model was then verified both theoretically \cite{Lieb} and experimentally \cite{parades} under the assumption that the mass of gas (number of particles) is sufficiently large.  Another physical example, where the quintic NLS plays a role is Bose-Einstein condensates where the three-body interatomic interactions (quintic nonlinearity) heavily dominates the two-body interactions (cubic nonlinearity) so that the latter becomes negligible \cite{abdul}.

If  $A_\Omega$ is constant and positive, then the quintic NLS is said to be defocusing (repulsive).  On the other hand, if $A_\Omega$ is constant and negative, then the quintic NLS is said to be focusing (attractive).  It is well-known that collapse of solutions is possible in the focusing case \cite{Glassey}.  One idea to manage the nonlinear effect in the focusing case is to insert certain physical changes to the system where a rapidly varying nonlinear effect with approximately zero mean can be created.  Then, one can hope to cancel out the adverse effects of the nonlinearity and convey the waves in a more stable form over an extended period of time.  This motivates us to study the mathematical properties of the model presented in \eqref{Model1a}-\eqref{Model1b} where $A_\Omega$ is no longer a constant and desirably a fast oscillating function of time.

\subsubsection*{Half-line model}  In physics, the half-line model posed in \eqref{Model2a}-\eqref{Model2c} with nonlinear Robin-type boundary condition (with say $A_\Omega\equiv 1$) is obtained from the one dimensional nonlinear Schrödinger equation with attractive point nonlinearity given by
\begin{equation}\label{diracnls}
  iu_t+u_{xx}+\delta |u|^ru=0,
\end{equation} where $\delta$ is the usual Dirac delta function, see for instance \cite{Holmer2015} for the relation between two models.  The power type nonlinearity on the boundary is interpreted as a jump condition at the origin given by $$\displaystyle\lim_{x\rightarrow 0^+}u_x(x,t)-\lim_{x\rightarrow 0^-}u_x(x,t)\equiv |u(0,t)|^ru(0,t)$$ for the free Schrödinger equation, which can be rewritten as \eqref{diracnls}.   This condition reduces to $$\displaystyle \lim_{x\rightarrow 0^+}u_x(x,t)\equiv |u(0,t)|^ru(0,t)$$ assuming the evolution is taking place only at the right half-space, i.e., $u(x,t)\equiv 0$ for $x<0$.

\eqref{diracnls} is used (say with $r=2$) to describe the evolution of the propagation of an electron subject to a vibrational impurity at $x=0,$ where the vibration can couple strongly and is completely enslaved to the electron \cite{Molina}.  Another use of this model is to describe the evolution of a wave travelling in a domain that contains a narrow strip of nonlinear (general Kerr-type) material, where the nonlinear strip is assumed to be much smaller than the typical wavelength \cite{yeh}.

In has been shown that collapse of solutions for \eqref{diracnls} or \eqref{Model2a}-\eqref{Model2c} (with $A_\Omega$ being a constant) is possible (\cite{Holmer2015}, \cite{AASKD}, \cite{VKTO}).  This again leads us to consider the same type of problem considered for the whole line problem.  That is, the study of the relationship between the collapse phenomena and the nonlinear oscillating effects.

\subsection{A few words on the previous results}
Ogawa-Tsutsumi \cite{OT} proved the following blow-up result for \eqref{Model1a}-\eqref{Model1b} with $A_\Omega$ being only a constant.

\begin{thm}[\cite{OT}]\label{OTThm}If $u_0\in H^1$ with $E(u_0)<0$, then there is $T>0$ such that $$\displaystyle\lim_{t\rightarrow T^-}\|u_x(t)\|_{L^2}=\infty.$$\end{thm}

Regarding the half-line problem, Ackleh-Deng \cite{AASKD} proved the following blow-up result for \eqref{Model2a}-\eqref{Model2c} with $A_\Omega$ being a constant.

\begin{thm}[\cite{AASKD}]\label{AASKDThm}If $u_0\in H^3(\mathbb{R}_+)$ with $E(u_0)<0$ and $r\ge 2$, then there is $T>0$ such that $$\displaystyle\lim_{t\rightarrow T^-}\|u_x(t)\|_{L^2(\mathbb{R}_+)}=\infty.$$\end{thm}

Moreover, \cite{AASKD} proves in the same context that the critical exponent associated with \eqref{Model2a}-\eqref{Model2c} is equal to $r=2$.  More precisely, all local solutions turn out to be also global for $r<2$, and one can construct blow-up solutions otherwise.

Recently, \cite{VKTO} generalized the blow-up result in \cite{AASKD} (with $a$ still being only a constant) to the case where \eqref{Model2a} is replaced by $iu_t-u_{xx}+k|u|^pu+i\gamma u=0$ ($k,p>0,\gamma\ge 0$).  In this work, the authors have studied the interaction between the nonlinear focusing Robin type boundary source, the nonlinear defocusing interior source, and the weak damping term.  \cite{VKTO} proved that there are blow-up solutions as long as the focusing type boundary nonlinearity is sufficiently stronger than the defocusing type interior nonlinearity ($r>\max\{2,p-2\}$).  The authors showed that although the damping has no effect on preventing the blow-up, it has a rapid stabilizing effect for global solutions where the rate of decay depends on the relation between the powers of nonlinearities.

\begin{rem}In Theorem \ref{OTThm} (resp. Theorem \ref{AASKDThm}), $E(u_0)$ denotes the initial energy where the energy functional is defined by \eqref{energy} (resp. \eqref{energy1}).\end{rem}

The blow-up of solutions in the presence of oscillating nonlinearities is more interesting than the constant coefficient case, because it is well-known that oscillating sources create a stabilizing effect by extending the life time of the solutions, see Remark \ref{StabEff}.  The blow-up of solutions for nonlinear Schrödinger equations with oscillating nonlinearities has only been shown in $H^1\cap L^2(|x|^2dx)$; see \cite{DO} and \cite{ZZ}. This in particular implies that the momentum $\displaystyle \int |x|^2|u(x)|^2dx$ was assumed to be finite. However, it is well-known from the general theory of nonlinear Schrödinger equations that $H^1$ is sufficient when we desire the local well-posedness for \eqref{Model1a}-\eqref{Model1b}. See for example, \cite{Kato87}.  Therefore, the aim here is to eliminate the weight assumption of $L^2(|x|^2dx)$ in the presence of an oscillating source.  Elimination of the finite momentum assumption was already achieved by Ogawa-Tsutsumi \cite{OT} in the case $A_\Omega$ is constant and the initial energy is strictly negative by using a compactly supported weight function $\varphi\in W^{3,\infty}(\mathbb{R})$ in virial identities.  This function was given by
\begin{equation}\label{varphidef}
 \varphi(x) =
  \begin{cases}
      \hfill x    \hfill & \text{ if $|x|\le 1$}, \\
      \hfill x-(x-1)^3 \hfill & \text{ if $1<x \le  1+\frac{1}{\sqrt{3}}$}, \\
      \hfill x-(x+1)^3 \hfill & \text{ if $-\left(1+\frac{1}{\sqrt{3}}\right)\le x< -1$}, \\
      \hfill \text{smooth}    \hfill & \text{ if $1+\frac{1}{\sqrt{3}}\le |x|<2$}, \\
      \hfill 0    \hfill & \text{ if $2\le |x|$}, \\
  \end{cases}
\end{equation} together with $\varphi'(x)\le 0$ for $|x|\ge 1+\frac{1}{\sqrt{3}}.$  We show that a similar result also holds when the initial energy is non-negative and the source has a time dependent coefficient (e.g., an oscillating function).  Construction of blow-up solutions for the problem \eqref{Model2a}-\eqref{Model2c} is similar.  For this purpose, we will use the weight function \begin{equation}\label{varphi2def}
                          \varphi=(x^2+x)e^{-x}.
                        \end{equation}  These type of weight functions were used for example in \cite{AASKD} and \cite{OT2}.

\subsection{Main result}

We associate the real line problem \eqref{Model1a}-\eqref{Model1b} with the energy function
\begin{equation}\label{energy}
  E(u(t)) = \|u_x(t)\|_{L^2}^2-\frac{a(2\Omega t)}{3}\|u(t)\|_{L^{6}}^{6}.
\end{equation}

If $a$ were constant, one would have the classical nonlinear Schrödinger equation, and it is well-known that the associated mass and energy are both conserved quantities in this case.  This fact plays an important role in the global analysis of nonlinear Schrödinger equations.  However, when $a$ is not constant, then energy needs not be conserved.  Indeed, \eqref{Model1a}-\eqref{Model1b} has the following identities for mass and rate of change of energy:

\begin{equation}\label{density}
    \|u(t)\|_{L^2} =  \|u_0\|_{L^2},
\end{equation}

\begin{equation}\label{Denergy1}
    \frac{d}{dt}[E(u(t))] =  -a'(2\Omega t)\frac{2\Omega }{3}\|u(t)\|_{L^{6}}^{6}.
\end{equation}

Therefore, once we assume that $A_\Omega(t_0)>0$ and $A_\Omega'(t)\ge 0$ in time interval $[t_0,t_0+T)$, where $T>0$, then if the energy is initially negative, it will continue to stay negative within the same time interval $[t_0,t_0+T)$.

Regarding the half-line problem \eqref{Model2a}-\eqref{Model2c}, the associated energy function is given by
\begin{equation}\label{energy1}
  E(u(t)) = \|u_x(t)\|_{L^2(\mathbb{R}_+)}^2-a(2\Omega t)\frac{2}{r+2}|u(t,0)|^{r+2}.
\end{equation}

 When $a$ is a differentiable function of time, we have the following rate of change for energy:
  \begin{equation}\label{Denergy2}
    E'(u(t)) =  -a'(2\Omega t)\frac{4\Omega }{r+2}|u(t,0)|^{r+2}.
\end{equation}

We observe from \eqref{energy1} and \eqref{Denergy2} that both the energy and its rate of change are influenced by what is happening at the corner point $x=0$.

\subsection*{Notation}
Let us introduce the following notation to simplify the statements of main theorems:
$\Psi(x)\equiv \int_0^x\varphi(y) dy$, $\displaystyle\lambda \equiv -2\Im\int_G \varphi u_0\bar{u}_0' dx$, $\displaystyle\alpha\equiv \int_G\Psi|u_0|^2dx$,   $$\displaystyle \delta=\frac{1}{2}\sqrt{\frac{3}{8A_{\Omega}(t_0+T)}}\left(\|\varphi'''\|_\infty+{\max\left\{\sqrt{3}, \frac{1}{2}\|\varphi''\|_\infty\right\}^2}\right)$$ for the whole line problem, $\delta\equiv \frac{1}{2}\|\varphi'''\|_\infty\|u_0\|_2^2$ for the half-line problem, $\displaystyle\beta\equiv (2E(u_0)+\delta)$, $\displaystyle \gamma\equiv \left(\frac{2}{\beta}\|u_0'\|_{L^2(\mathbb{R})}+1\right)$, $\displaystyle \theta_{\pm}\equiv \frac{-\lambda\pm\sqrt{\lambda^2-4\alpha\beta}}{2\beta}$, $\theta_0=-\alpha/\lambda$,  $G=\mathbb{R}$ for the whole line problem, $G=\mathbb{R}_+$ for the half line problem.  $\varphi$ is chosen as in \eqref{varphidef} for the whole line problem and as in \eqref{varphi2def} for the half line problem.

Now, we state our main results for the whole line problem and the half-line problem (with $r\ge 2$), respectively:

\begin{thm}[Whole line problem]\label{mainresult-wl} Let $t_0\in\mathbb{R}$, $\Omega>0$ be fixed, $A_\Omega(t)= a(2\Omega t)$ for some real valued smooth function $a(t)$ such that $A_\Omega(t_0)>0$ with $A_\Omega'(t)\ge 0$ in $[t_0,t_0+T)$ for some $T>0$. If $u_0\in H^1(\mathbb{R})$ satisfies one of the following conditions: \begin{itemize}

                        \item[(i)] $E(u_0)> -\delta/2$, $\lambda < 0$, $\lambda^2>4\alpha\beta$, $T>\theta_-$, $\displaystyle \alpha+\beta\theta_{-}^2<\frac{1}{2}\left(\frac{3}{8A_\Omega(t_0+T)}\right)^{\frac{1}{2}}$,
                        \item[(ii)] $E(u_0)<-\delta/2$, $T>\theta_+$, $\displaystyle 2\alpha\gamma<\left(\frac{3}{8A_\Omega(t_0+T)}\right)^{\frac{1}{2}}$,
                        \item[(iii)] $E(u_0)= -\delta/2$, $\lambda < 0$, $T>\theta_0$, $\displaystyle \alpha<\frac{1}{2}\left(\frac{3}{8A_\Omega(t_0+T)}\right)^{\frac{1}{2}}$,
                      \end{itemize} then the corresponding solution of \eqref{Model1a}-\eqref{Model1b} must blow-up in finite time (more precisely before $t$ reaches $t_0+T$) in the sense that there is a time $T^*\in (0,T)$ which satisfies $$\displaystyle \lim_{t\uparrow t_0+T^*}\|u_x(t)\|_{L^2(\mathbb{R})} = \infty.$$
\end{thm}

\begin{thm}[Half line problem]\label{mainresult-hl} Let $t_0\in\mathbb{R}$, $\Omega>0$ be fixed, $A_\Omega(t)= a(2\Omega t)$ for some real valued smooth function $a(t)$ such that $A_\Omega(t_0)>0$ with $A_\Omega'(t)\ge 0$ in $[t_0,t_0+T)$ for some $T>0$. If $u_0\in H^1(\mathbb{R}_+)$ satisfies one of the following conditions: \begin{itemize}
                        \item[(i)] $E(u_0)> -\delta/2$, $\lambda < 0$, $\lambda^2>4\alpha\beta$, $T>\theta_-$,
                        \item[(ii)] $E(u_0)<-\delta/2$, $T>\theta_+$,
                        \item[(iii)] $E(u_0)= -\delta/2$, $T>\theta_0$, $\lambda < 0$,
                      \end{itemize} then the corresponding solution of \eqref{Model2a}-\eqref{Model2c} must blow-up in finite time (more precisely before $t$ reaches $t_0+T$) in the sense that there is a time $T^*\in (0,T)$ which satisfies $$\displaystyle \lim_{t\uparrow t_0+T^*}\|u_x(t)\|_{L^2(\mathbb{R}_+)} = \infty.$$ \end{thm}
\subsection{Scaling argument} Theorems \ref{mainresult-wl} and \ref{mainresult-hl} give sufficient conditions on the initial datum $u_0$ for blow-up to occur once $t_0,\Omega,$ and $A_\Omega(\cdot)$ are given.  The natural question to ask is whether the given sufficient conditions on $u_0$ are void and one might wonder the answer to the following general question:

\begin{prob}
  Given $t_0$, $\Omega$, $A_\Omega(\cdot)$ and $T$ as in Theorem \ref{mainresult-wl} (or Theorem \ref{mainresult-hl}), can you find an initial datum $u_0$ for each sufficient condition given in the theorems that guarantees the blow-up of the corresponding solution?
\end{prob}

The answer to the above question is `yes' for sufficient conditions (ii) and (iii) in both theorems and one can see this through a scaling argument. Indeed, let $t_0$, $\Omega$, and $A_\Omega(\cdot)$ be such that $A_\Omega(t_0)>0$ with $A_\Omega'(t)\ge 0$ in $[t_0,t_0+T)$ for some $T>0$, and say we consider the whole line problem (Theorem \ref{mainresult-wl}). Let us for instance start with taking an initial datum $u_0\in H^1(\Omega)$ with negative initial energy ($E(u_0)<0$). As a second step, consider the scaling given by $\displaystyle u_{0\mu}^\rho(x):=\frac{\mu}{\sqrt{\rho}}u_0\left(\frac{x}{\rho}\right)$ with $\mu>0$ and $\rho>0$. It follows by a straightforward change of variables that this scaling argument changes the $L^2-$norm proportional to $\mu$ while changing the energy proportional to $\mu^2/\rho^2$.  More precisely, one has $\displaystyle\|u_{0\mu}^\rho\|_{L^2}=\mu\|u_0\|_{L^2}$ and $\displaystyle E(u_{0\mu}^\rho)=\frac{\mu^2}{\rho^2}E(u_0)$. Therefore, for $\mu=\left(-\frac{\delta\rho^2}{E(u_0)}\right)^\frac{1}{2}$, one has $E(u_{0\mu}^\rho)=-\delta<-\delta/2$, as desired in Theorem \ref{mainresult-wl} (ii).  Note that the same scaling has the effect $\alpha\sim \mu^2$, while other parameters satisfy $\displaystyle\lambda \sim \frac{\mu^2}{\rho}$, $\displaystyle\beta\sim \frac{\mu^2}{\rho^2}$, $\displaystyle\theta_+ \sim \rho$, and $\displaystyle\gamma\sim \text{ constant}$.  These imply that for sufficiently small but fixed $\rho>0$, the conditions $T>\theta_+$, $\displaystyle 2\alpha\gamma<\left(\frac{3}{8A_\Omega(t_0+T)}\right)^{\frac{1}{2}}$ can be guaranteed to hold since $\theta_+$ and $2\alpha\gamma$ can be made as small as desired by choosing $\rho$ small. Hence, we have just shown that one can always construct an initial datum satisfying all the assumptions in Theorem \ref{mainresult-wl} (ii).

Regarding the sufficient condition in Theorem \ref{mainresult-wl} (iii), we again start with an initial datum $u_0\in H^1(\Omega)$ with negative initial energy ($E(u_0)<0$), and we consider the similar scaling given by $\displaystyle u_{0\mu}^\rho(x):=\frac{\mu}{\sqrt{\rho}}u_0\left(\frac{x}{\rho}\right)$ with $\mu>0$ and $\rho>0$.  Note that the parameters satisfy $\displaystyle\alpha\sim \mu^2$ and $\displaystyle E(u_{0\mu}^\rho) = \frac{\mu^2}{\rho^2}E(u_0)$, $\theta_0\sim \rho$. Therefore, taking $\displaystyle \mu=\left(-\frac{\delta\rho^2}{2E(u_0)}\right)^{\frac{1}{2}}$ and $\rho$ suitably small, we can make $\alpha,\theta_0$ as small as we wish so that the conditions $\displaystyle \alpha<\frac{1}{2}\left(\frac{3}{8A_\Omega(t_0+T)}\right)^{\frac{1}{2}}$ and $T>\theta_0$ are certainly satisfied. By the choice of $\mu$, we also have $E(u_{0\mu}^\rho)=-\delta/2$. Hence, $u_{0\mu}^\rho$ becomes a sought after initial datum for which the corresponding solution blows up.

The scaling arguments given in the two paragraphs above work pretty much in the same way also for the half-line problem  (Theorem \ref{mainresult-hl} (ii)-(iii)).

On the other hand, the scaling does not seem to provide a proof for the argument that one can always construct some initial data satisfying the sufficient conditions given in Theorems \ref{mainresult-wl} (i) and \ref{mainresult-hl} (i) because  for instance both the left and right hand side of the inequality $\lambda^2>4\alpha\beta$ are affected the same way from the scaling.  However, from the calculations relevant to the construction of the numerical example given in Section \ref{NumSec}, we can say that the scaling still works in practice.  Intuitively, this is due to the fact that physically the likelihood of the collapse gets bigger as the nonlinear effect gets larger which is exactly what is achieved via scaling.  Therefore, this method can still be used as a heuristic also in the case $E(u_0)>-\delta/2$ to find a suitable initial datum for which the corresponding solution blows up.

\begin{rem}
  \begin{itemize}
    \item[(1)] It turns out that no condition as $\displaystyle\alpha+\beta\theta_{-}^2<\frac{1}{2}\left(\frac{3}{8A_\Omega(t_0+T)}\right)^{\frac{1}{2}}$ in Assumption (i) or as $\displaystyle 2\alpha\gamma<\left(\frac{3}{8A_\Omega(t_0+T)}\right)^{\frac{1}{2}}$ in Assumption (ii) is necessary in the case of the half-line problem.  Compare Theorems \ref{mainresult-wl} and \ref{mainresult-hl} above.
    \item[(2)] In the case $E(u_0)\le -\delta/2$, it is actually possible to remove the conditions $T>\theta_+$, $T>\theta_0$, $\displaystyle 2\alpha\gamma<\left(\frac{3}{8A_\Omega(t_0+T)}\right)^{\frac{1}{2}}$, and $\displaystyle \alpha<\frac{1}{2}\left(\frac{3}{8A_\Omega(t_0+T)}\right)^{\frac{1}{2}}$  in Assumptions (ii)-(iii) of Theorem \ref{mainresult-wl} (also Theorem \ref{mainresult-hl}). This follows by using the scaling argument explained above and the fact that $u_{0\mu}^\rho$ blows up if and only if $u_0$ blows up.

  \end{itemize}
\end{rem}
\subsection{Comparison with constant coefficient case}The blow-up results given in Theorems \ref{mainresult-wl} and \ref{mainresult-hl} are interesting when $A_\Omega$ is non-constant.  In order to see this, one can for example take $A_\Omega(t)$ as an oscillating function of $t$. Then these oscillations will expand the life span of solutions \cite{DO}.    This can be easily seen also from the proof of the local well-posedness. For instance, the example given in Remark \ref{StabEff} shows that the life-span of solutions must at least double if one takes $A_\Omega(t)=\cos^2(\Omega t)$ with large $\Omega$ compared to the case $A_\Omega$ is a constant.  Therefore, the solutions which normally blow-up at some particular time $T_0$ for the constant coefficient case will no longer blow-up at least up to the time $2T_0$.  Thus, the question is whether the stabilizing effect of these oscillations is strong enough to turn any blowing-up solution into a global one.  We answer this question in Theorems \ref{mainresult-wl} and \ref{mainresult-hl} in the negative. We show that one can always find some suitable initial data for which the corresponding solution will be steered to infinity in $H^1$ norm as long as there is a little bit of chance for the time dependent coefficient to keep its strict positiveness continuously in a subinterval of time.

The blow-up of solutions in the work of Ogawa-Tsutsumi \cite{OT} is given under the condition $E(u_0)<0$ (strictly negative initial energy).  Theorems \ref{mainresult-wl} and \ref{mainresult-hl} give more general sets of conditions which must be satisfied by $u_0$, because we prove that blow-up of solutions is also possible in the case of non-negative initial energy if $u_0$ also satisfies additional restrictions.  The associated conditions \[\frac{\Im\int_G \varphi u_0\bar{u}_0' dx\mp\sqrt{\left(\Im\int_G \varphi u_0\bar{u}_0' dx\right)^2-(2E(u_0)+\delta)\int_G\Psi |u_0|^2dx}}{2E(u_0)+\delta}<T\] that we give in Theorems \ref{mainresult-wl} and \ref{mainresult-hl} are automatically satisfied in the work \cite{OT} since they only consider the case $T=\infty$.  Note that the last condition enables one to choose the correct initial data to obtain solutions which blow-up before time $T$ independent of how small $T$ is.

\section{Interior Oscillations}
In this subsection, we prove Theorem \ref{mainresult-wl} by using the method in \cite{OT} taking into account that $A_{\Omega}$ is now non-constant and the energy is allowed to be non-negative.

\subsection{Local Well-Posedness}\label{lwp}
The proof of the local well-posedness of $H^1$ and $H^2$ type solutions for \eqref{Model1a}-\eqref{Model1b} is classical and one has the following result.
\begin{lem}[see \cite{Cazenave}] Let $u_0\in H^1(\mathbb{R})$ ($H^2(\mathbb{R})$).  Then \eqref{Model1a}-\eqref{Model1b} is locally well-posed in $H^1$ ($H^2$), i.e., there exists a unique solution  $u\in C(t_0, t_0+T_0;H^1)$ $(u\in C(t_0, t_0+T_0;H^2))$ for some $T_0>0$, where $T_0$ depends on the respective norm of $u_0$.  Moreover, $u$ also satisfies the $H^1$ blow-up alternative.\end{lem}

\begin{rem}[Stabilizing Effect]\label{StabEff}It is worth mentioning that oscillating coefficients create a stabilizing effect by extending the life time of the solution \cite{DO}. For example, let's say $A_\Omega(t)=\cos^2(\Omega t)$, and $t_0=0$. Then $$\int_0^{T_0}\cos^2(\Omega t)dt=\frac{{T_0}}{2}+\frac{\sin(2\Omega {T_0})}{4\Omega}\sim \frac{{T_0}}{2}$$ for large $\Omega.$   Taking this approximation into account, one can deduce that the lifetime of the solution is almost doubled compared to the case $\Omega=0.$  For more details on the stabilizing effect, see \cite[Prop. 4, Prop. 5]{DO}.\end{rem}

\subsection{Virial Identities, and Estimates} We have the following virial identities.
\begin{lem}\label{virial}Let $u\in C(t_0,t_0+T_0; H^1)$ be a solution of \eqref{Model1a}-\eqref{Model1b}.  Then,
\begin{multline}
  -\Im \int \varphi u\bar{u}_x dx + \Im \int \varphi u_0\bar{u}_0' dx \\
  = \int_{t_0}^t\left[2\int\varphi'|u_x|^2dx-\frac{2a(2\Omega s)}{3}\int \varphi'|u|^{6}dx-\frac{1}{2}\int\varphi'''|u|^2dx\right]ds,\\
  =\int_{t_0}^t\left[2E(u_0)-2\int_{|x|\ge {1}}(1-\varphi')|u_x|^2dx+\frac{2a(2\Omega s)}{3}\int_{|x|\ge {1}}(1-\varphi')|u|^{6}dx\right]ds\\
  +\int_{t_0}^t\left[-\frac{2\Omega }{3}\int_{t_0}^sa'(2\Omega \tau)\|u(\tau)\|_{L^{6}}^{6}d\tau-\frac{1}{2}\int\varphi'''|u|^2dx\right]ds
\end{multline}

and

$$\int \Psi |u|^2 dx = \int \Psi|u_0|^2dx -2\int_{t_0}^t\Im \int \varphi u\bar{u}_xdxds$$ for $t\in [t_0,t_0+T_0)$.
\end{lem}

\begin{proof}  In this proof, we will give formal calculations. However, we can always justify these by approximating the initial data with smooth functions and running the multipliers on the corresponding regularized solutions.

Now if we multiply \eqref{Model1a} by $\varphi \bar{u}_x$ and take the real part, we obtain
\begin{equation}\label{Est01}
  \Re \int iu_t\varphi \bar{u}_x dx = -\Re \int u_{xx}\varphi \bar{u}_xdx-\Re \int a(2\Omega t)|u|^4u \varphi \bar{u}_xdx
\end{equation} where

$$\Re \int iu_t\varphi \bar{u}_x dx = - \Im \int \varphi u_t\bar{u}_x dx = -\frac{d}{dt}\Im \int \varphi u\bar{u}_x dx +\Im \int \varphi u\bar{u}_{xt} dx$$
$$=-\frac{d}{dt}\Im \int \varphi u\bar{u}_x dx - \Im \int \varphi' u\bar{u}_{t} dx - \Im \int \varphi u_x\bar{u}_{t} dx.$$  Therefore,
\begin{equation}\label{Est02}\Re \int iu_t\varphi \bar{u}_x dx =  -\frac{1}{2}\frac{d}{dt}\Im \int \varphi u\bar{u}_x dx -\frac{1}{2}\Im \int \varphi' u\bar{u}_{t} dx.\end{equation}

The first term at the right hand side of \eqref{Est01} is
$$-\Re \int u_{xx}\varphi \bar{u}_xdx = \int \varphi'|u_x|^2dx + \Re\int\varphi u_x \bar{u}_{xx}dx,$$ from which it follows that
\begin{equation}\label{Est03}-\Re \int u_{xx}\varphi \bar{u}_xdx = \frac{1}{2} \int \varphi'|u_x|^2dx.\end{equation}

The second term at the right hand side of \eqref{Est01} is
\begin{equation}\label{Est04}-\Re \int a(2\Omega t)|u|^4u \varphi \bar{u}_xdx = \frac{a(2\Omega t)}{6}\int \varphi'|u|^{6} dx.\end{equation}  Combining \eqref{Est01}-\eqref{Est04} we obtain
\begin{equation}\label{Est05}-\frac{d}{dt}\Im \int \varphi u\bar{u}_x dx -\Im \int \varphi' u\bar{u}_{t} dx = \int \varphi'|u_x|^2dx + \frac{2a(2\Omega t)}{6}\int \varphi'|u|^{6} dx.\end{equation}

Multiplying the complex conjugate of \eqref{Model1a} by $\varphi' u$ and taking the real part
$$-\Re \int i\bar{u}_t \varphi' u dx = \Re \int -\bar{u}_{xx}\varphi' u dx - \Re \int a(2\Omega t)|u|^4\bar{u}\varphi' u dx,$$ where the left hand side is equal to \begin{equation}\label{Est06}\Im \int \varphi' u\bar{u}_{t} dx,\end{equation} and the right hand side is equal to
\begin{equation}\label{Est07}-\frac{1}{2}\int\varphi'''|u|^2dx + \int\varphi'|u_x|^2dx - a(2\Omega t)\int\varphi'|u|^{6}dx.\end{equation}

On the other hand, integrating \eqref{Denergy1} over the time interval $(t_0,t)$, we obtain
\begin{multline}\label{Ienergy}
\int_{|x|<1}|u_x|^2dx-\frac{a(2\Omega t)}{3}\int_{|x|<1}|u|^{6}dx \\
= E(u_0) - \int_{|x|\ge 1}|u_x|^2dx+\frac{a(2\Omega t)}{3}\int_{|x|\ge 1}|u|^{6}dx-\frac{2\Omega }{3}\int_{t_0}^ta'(2\Omega s)\|u(s)\|_{L^{6}}^{6}ds.
\end{multline}

Combining \eqref{Est05}-\eqref{Est07}, using \eqref{Ienergy}, and integrating over the time interval $(t_0,t),$ we get the following identity:
\begin{multline}\label{varphi01}
  -\Im \int \varphi u\bar{u}_x dx + \Im \int \varphi u_0\bar{u}_0' dx \\
  = \int_{t_0}^t\left[2\int\varphi'|u_x|^2dx-\frac{2a(2\Omega s)}{3}\int \varphi'|u|^{6}dx-\frac{1}{2}\int\varphi'''|u|^2dx\right]ds\\
  =\int_{t_0}^t\left[2E(u_0)-2\int_{|x|\ge 1}(1-\varphi')|u_x|^2dx+\frac{2a(2\Omega s)}{3}\int_{|x|\ge 1}(1-\varphi')|u|^{6}dx\right]ds\\
  +\int_{t_0}^t\left[-\frac{2\Omega }{3}\int_{t_0}^sa'(2\Omega \tau)\|u(\tau)\|_{L^{6}}^{6}d\tau-\frac{1}{2}\int\varphi'''|u|^2dx\right]ds.
\end{multline}

Multiplying the complex conjugate of \eqref{Model1a} by $\Psi u$, taking the imaginary parts, integrating over $\mathbb{R}\times (0,t)$, we obtain

$$\int \Psi |u|^2 dx = \int \Psi|u_0|^2dx -2\int_{t_0}^t\Im \int \varphi u\bar{u}_xdxds.$$
\end{proof}

Now we will prove the estimate given by the lemma below.

\begin{lem}\label{interiorest} Let $u\in C(t_0,t_0+T; H^1)$ be a solution of \eqref{Model1a}-\eqref{Model1b} and $a'>0$ in $[t_0,t_0+T)$.  Then, for $$\displaystyle \delta=\frac{1}{2}\sqrt{\frac{3}{8A_{\Omega}(t_0+T)}}\left(\|\varphi'''\|_\infty+{\max\left\{\sqrt{3}, \frac{1}{2}\|\varphi''\|_\infty\right\}^2}\right),$$ one has
\begin{equation}\label{varphi2lem}
  -\Im \int \varphi u\bar{u}_x dx\le - \Im \int \varphi u_0\bar{u}_0' dx + (2E(u_0)+\delta) (t-t_0)
\end{equation} on $[t_0,t_0+T)$ provided $\displaystyle \|u(t)\|_{L^2(|x|\ge 1)}^4<\frac{3}{8A_{\Omega}(t_0+T)}$.
\end{lem}

\begin{proof}

For $x\ge 1$ we have

\begin{multline}\label{xge1eps}
(1-\varphi(x)')^{\frac{1}{2}}|u(x)|^2 = -\int_x^\infty \left[(1-\varphi(y)')^{\frac{1}{2}}|u(y)|^2\right]_ydy \\
\le \frac{1}{2}\int_{y\ge 1}(1-\varphi')^{-\frac{1}{2}}|\varphi''||u|^2dy+2\int_{y\ge 1} (1-\varphi')^{\frac{1}{2}}|u||u_x|dy\\
\le C_0\|u\|_{L^2(|x|\ge 1)}^2+2\|u\|_{L^2(|x|\ge 1)}\|(1-\varphi')^{\frac{1}{2}}u_x\|_{L^2(|x|\ge 1)}
\end{multline} where $\displaystyle C_0=\max\{\sqrt{3}, \frac{1}{2}\|\varphi''\|_\infty\}$. (A similar estimate also holds for $x\le -1$; the details are omitted.)  Hence, one has

\begin{equation}\label{linf01}\|(1-\varphi')|u|^4\|_{L^\infty(|x|\ge 1)}\le 2C_0^2\|u\|_{L^2(|x|\ge 1)}^4+8\|u\|_{L^2(|x|\ge 1)}^2\|(1-\varphi')^{\frac{1}{2}}u_x\|_{L^2(x\ge 1)}^2.\end{equation}

Using the assumption that $a'>0$ in $[t_0,t_0+T)$ and \eqref{linf01} in \eqref{varphi01}, we get the estimate
\begin{multline}\label{varphi2}
  -\Im \int \varphi u\bar{u}_x dx + \Im \int \varphi u_0\bar{u}_0' dx \\
  \le \int_{t_0}^t\left[2E(u_0)+C_1\|u(s)\|_{L^2(|x|\ge 1)}^2+C_2\|u(s)\|_{L^2(|x|\ge 1)}^6\right]ds\\
  -2\int_{t_0}^t\left(1-\frac{8C_3}{3}\|u(s)\|_{L^2(|x|\ge 1)}^4\right)\int_{|x|\ge 1}(1-\varphi')|u_x|^2dxds
\end{multline} where $\displaystyle C_1=\frac{1}{2}\|\varphi'''\|_\infty$, $\displaystyle C_2=\frac{4}{3}C_0^2A_\Omega(t_0+T)$, and $\displaystyle C_3=A_\Omega(t_0+T)$.  Therefore, the result follows with $\displaystyle \delta=\sqrt{\frac{3}{8C_3}}\left(C_1+\frac{C_0^2}{2}\right).$
\end{proof}

\subsection{Proof of collapse for the real-line problem}\label{Theorem1Details}

\subsubsection*{Step 1.} We first prove the result by assuming that the condition $$\displaystyle \|u(t)\|_{L^2(|x|\ge 1)}^4<\frac{3}{8A_{\Omega}(t_0+T)}$$ is satisfied on $[0,T]$.  The general case can always be reduced to this situation by using the given assumptions on $u_0$ in Theorem \ref{mainresult-wl}, see Lemma \ref{contra} below.

Now, assume to the contrary that the solution $u(t)$ of \eqref{Model1a}-\eqref{Model1b} exists for all $t\in [t_0,t_0+T)$.  Using Lemmas \ref{virial} and \ref{interiorest}, we get
\begin{equation}\label{Bestest}\int \Psi |u|^2 dx \le \int \Psi|u_0|^2dx -2(t-t_0)\Im\int \varphi u_0\bar{u}_0' dx +(2E(u_0)+\delta)(t-t_0)^2\end{equation} for $t\in [t_0,t_0+T)$.

In order to shorten the notation, we adapt to the notation given at the beginning of the paper and we also set $\tau = t-t_0$.
By \eqref{Bestest}, we have
\begin{equation}
  0\le \int \Psi |u|^2 dx\le \alpha+\lambda\tau+\beta\tau^2 \equiv p(\tau), \tau\in[0,T).
\end{equation}

Let us analyze different cases for the quadratic polynomial $p(\tau)$ depending on the signs of $\lambda$ and $\beta$.

\subsubsection{Case: $\beta> 0$ (satisfied if $E(u_0)> -\delta/2$)}  In this case, $p(\tau)$ can only be negative if $\lambda<0$, $\lambda^2-4\alpha\beta>0$ and as soon as $\tau>\frac{-\lambda- \sqrt{\lambda^2-4\alpha\beta}}{2\beta}.$  This yields a contradiction if we choose $u_0$ in such a way that $T>\frac{-\lambda- \sqrt{\lambda^2-4\alpha\beta}}{2\beta}$, because then there will be a time $\tau\in (0,T)$ such that $0\le \int\Psi |u|^2 dx<0$, which is absurd.

\subsubsection{Case: $\beta<0$ (satisfied if $E(u_0)<-\delta/2$)} In this case $p(\tau)$ is negative as soon as $\frac{-\lambda+\sqrt{\lambda^2-4\alpha\beta}}{2\beta}<\tau.$  Again, this will contradict our assumption.

\subsubsection{Case: $\beta=0$ (satisfied if $E(u_0)=-\delta/2$)} In this case, $p(\tau)$ can only be negative if $\lambda<0$ and as soon as $\tau>-\alpha/\lambda.$

\subsubsection*{Step 2.} Now, we deal with the general situation and show that it can always be reduced to the case of Step 1 under the given assumptions on $u_0.$  To this end, it is enough to prove the following.

\begin{lem}\label{contra} If $u_0$ satisfies the conditions of Theorem \ref{mainresult-wl}, then there exists $T'>\theta_{-}$ if $\beta>0$, $T'>\theta_{+}$ if $\beta<0$, and $T'>\theta_{0}$ if $\beta=0$ such that the corresponding solution satisfies $$\displaystyle \|u(t)\|_{L^2(|x|\ge 1)}<\left(\frac{3}{8A_\Omega(t_0+T)}\right)^\frac{1}{4}$$ on $[t_0,t_0+T']$.
\end{lem}

\begin{proof} First consider the case $\beta>0$, i.e., the assumptions in part (i) of Theorem \ref{mainresult-wl}.  Note that if $$\displaystyle \alpha+\beta\theta_{-}^2<\frac{1}{2}\left(\frac{3}{8A_\Omega(t_0+T)}\right)^{\frac{1}{2}},$$ then
$$\|u_0\|_{L^2(|x|\ge 1)}^2\le 2\int\Psi|u_0|^2dx<\left(\frac{3}{8A_\Omega(t_0+T)}\right)^{\frac{1}{2}}.$$  Now, we set $$T'=\sup\left\{\tau\,|\,\tau\ge 0, \|u(t_0+\tau)\|_{L^2(|x|\ge 1)}\le \left(\frac{3}{8A_\Omega(t_0+T)}\right)^{\frac{1}{4}}\right\}.$$  Clearly $T'>0$ by continuity.  Now, if $T'\le \theta_{-}$, then $$\|u(t_0+T')\|_{L^2(|x|\ge 1)}=\left(\frac{3}{8A_\Omega(t_0+T)}\right)^{\frac{1}{2}}.$$ On the other hand, we know that for $\tau\in [0,T']$,
$$\|u(t_0+\tau )\|_{L^2(|x|\ge 1)}^2\le 2\int\Psi|u|^2dx\le 2(\alpha+\beta\tau^2)\le 2(\alpha+\beta\theta_{-}^2)<\left(\frac{3}{8A_\Omega(t_0+T)}\right)^{\frac{1}{2}}.$$ Hence, we arrive at a contradiction, and it must be true that $T'>\theta_{-}$.  Therefore, $$\|u(t_0+\tau)\|_{L^2(|x|\ge 1)}^2\le \left(\frac{3}{8A_\Omega(t_0+T)}\right)^{\frac{1}{2}}$$ for all $\tau\in [0,T'].$

The proof of the Lemma for the case $\beta=0$ is similar.  The case $\beta<0$ for which $u_0$ satisfies the assumptions in part (ii) of Theorem \ref{mainresult-wl} can be done as in \cite{OT}, and therefore omitted here.
\end{proof}
\section{Boundary Oscillations}

\subsection{Local Well-Posedness} Local well-posedness of $H^1$ solutions (under the assumption $u_0\in H^3(\mathbb{R}_+)$) was obtained by \cite{AASKD} in the case $A_\Omega$ is constant.    In \cite{AASKD} the initial data was assumed to be too smooth compared to the regularity of the solutions obtained.  It is well-known from the theory of the linear Schrödinger equation that solutions are of the same class as the initial data.  Recently, we improved this well-posedness result (see \cite{BO}) where the main equation also included a nonlinear source term of the form $k|u|^pu$. By using the fact that $A_\Omega$ is smooth , the lemma below follows as a corollary to \cite[Theorem 1.1]{BO} at this point.

\begin{lem}[Local well-posedness]\label{LocalWellP} Let $T>0$ be arbitrary, $s\in \left(\frac{1}{2},\frac{7}{2}\right) -\left\{\frac{3}{2}\right\}$,  $r>0$, $u_0\in H^s(\mathbb{R_+})$, and ${u_0'(0)=-A_\Omega(t_0)|u_0(0)|^ru_0(0)}$ whenever $s>\frac{3}{2}$.  We also assume that $r>\frac{2s-1}{4}$ if $r$ is an odd integer and $[r]\ge \left[\frac{2s-1}{4}\right]$ if $r$ is non-integer.  Then, the following hold true.
\begin{itemize}
  \item[(i)] Local Existence and Uniqueness: There exists a unique local solution $u\in X_{T_0}^s$ of \eqref{Model2a}-\eqref{Model2c} for some $T_0=T_0\left(\|u_0\|_{H^s(\mathbb{R}_+)}\right)\in (0,T]$, where $X_{T_0}^s$ is the set of those elements in $$C([t_0,t_0+T_0];H^s(\mathbb{R}_+))\cap C(\mathbb{R}_+^x;H^{\frac{2s+1}{4}}(t_0,t_0+T_0))$$ that are bounded with respect to the norm ${\|\cdot\|_{X_{T_0}^s}}$.   This norm is defined by $$\|u\|_{X_{T_0}^s}:=\sup_{t\in[t_0,t_0+{T_0}]}\|u(\cdot,t)\|_{H^s(\mathbb{R_+})}+\sup_{x\in\mathbb{R}_+}\|u(x,\cdot)\|_{H^{\frac{2s+1}{4}}(t_0,t_0+{T_0})}.$$
  \item[(ii)] Continuous Dependence: If $B$ is a bounded subset of $H^s(\mathbb{R}_+)$, then there is $T_0>0$ (depends on the diameter of $B$) such that the flow $u_0\rightarrow u$ is Lipschitz continuous from $B$ into $X_{T_0}^s$.
  \item[(iii)] Blow-up Alternative: If $S$ is the set of all $T_0\in (0,T]$ such that there exists a unique local solution in $X_{T_0}^s$, then whenever $\displaystyle T_{max}:=\sup_{T_0\in S}T_0<T$, it must be true that ${\displaystyle\lim_{t\uparrow t_0+T_{max}}\|u(t)\|_{H^s(\mathbb{R}_+)}=\infty}$.
\end{itemize}
 \end{lem}

\subsection{Virial Identities}
We have the following virial identities.
\begin{lem}\label{virialB}Let $u\in C(t_0,t_0+T_0; H^1)$ be a solution of \eqref{Model2a}-\eqref{Model2c}.  Then,
\begin{multline}
    -\Im \int_0^\infty \varphi u\bar{u}_x dx + \Im \int_0^\infty \varphi u_0\bar{u}_0' dx \\
  = -a(2\Omega t)|u(t,0)|^{r+2}-\frac{1}{2}\int_0^\infty\varphi'''|u|^2dx + 2\int_0^\infty\varphi'|u_x|^2dx
\end{multline}

and

$$\int_0^\infty \Psi |u|^2 dx = \int_0^\infty \Psi|u_0|^2dx -2\int_{t_0}^t\Im \int_0^\infty \varphi u\bar{u}_xdxds$$ for $t\in [t_0,t_0+T_0)$.
\end{lem}

\begin{proof}

Now if we multiply \eqref{Model2a} by $\varphi \bar{u}_x$ and take the real part, we obtain
\begin{equation}\label{Est01B}
  \Re \int_0^\infty iu_t\varphi \bar{u}_x dx = -\Re \int_0^\infty u_{xx}\varphi \bar{u}_xdx
\end{equation} where

$$\Re \int_0^\infty iu_t\varphi \bar{u}_x dx = - \Im \int_0^\infty \varphi u_t\bar{u}_x dx = -\frac{d}{dt}\Im \int_0^\infty \varphi u\bar{u}_x dx +\Im \int_0^\infty \varphi u\bar{u}_{xt} dx$$
$$=-\frac{d}{dt}\Im \int_0^\infty \varphi u\bar{u}_x dx - \Im \int_0^\infty \varphi' u\bar{u}_{t} dx - \Im \int_0^\infty \varphi u_x\bar{u}_{t} dx.$$

Therefore,
\begin{equation}\label{Est02B}\Re \int_0^\infty iu_t\varphi \bar{u}_x dx =  -\frac{1}{2}\frac{d}{dt}\Im \int_0^\infty \varphi u\bar{u}_x dx -\frac{1}{2}\Im \int_0^\infty \varphi' u\bar{u}_{t} dx.\end{equation}

Note that in the above calculation boundary terms vanished since $\varphi(0)=0$ and $\varphi$ vanishes at infinity.

The term at the right hand side of \eqref{Est01B} is
$$-\Re \int_0^\infty u_{xx}\varphi \bar{u}_xdx = \int_0^\infty \varphi'|u_x|^2dx + \Re\int_0^\infty\varphi u_x \bar{u}_{xx}dx,$$ from which it follows that
\begin{equation}\label{Est03B}-\Re \int_0^\infty u_{xx}\varphi \bar{u}_xdx = \frac{1}{2} \int_0^\infty \varphi'|u_x|^2dx.\end{equation}

Combining \eqref{Est01B}-\eqref{Est03B} we obtain
\begin{equation}\label{Est05B}-\frac{d}{dt}\Im \int_0^\infty \varphi u\bar{u}_x dx -\Im \int_0^\infty \varphi' u\bar{u}_{t} dx = \int_0^\infty \varphi'|u_x|^2dx.\end{equation}

Multiplying the complex conjugate of \eqref{Model2a} by $\varphi' u$ and taking the real part,
$$-\Re \int_0^\infty i\bar{u}_t \varphi' u = \Re \int_0^\infty -\bar{u}_{xx}\varphi' u dx,$$ where the left hand side is equal to \begin{equation}\label{Est06}\Im \int_0^\infty \varphi' u\bar{u}_{t} dx\end{equation} and the right hand side is equal to
\begin{equation}\label{Est07B}-a(2\Omega t)|u(t,0)|^{r+2}-\frac{1}{2}\int_0^\infty\varphi'''|u|^2dx + \int_0^\infty\varphi'|u_x|^2dx .\end{equation}

Note that in the above calculation, we use the facts that $\varphi'(0)=1, \varphi''(0)=0$ and $\varphi$ vanishing at infinity in order to treat the boundary terms.

Combining \eqref{Est05B}-\eqref{Est07B} and integrating over the time interval $(t_0,t)$, we get the following identity:
\begin{multline}\label{varphi01B}
  -\Im \int_0^\infty \varphi u\bar{u}_x dx + \Im \int_0^\infty \varphi u_0\bar{u}_0' dx \\
  = \int_{t_0}^t\left(-a(2\Omega s)|u(s,0)|^{r+2}-\frac{1}{2}\int_0^\infty\varphi'''|u|^2dx + 2\int_0^\infty\varphi'|u_x|^2dx\right)ds.
\end{multline}

Multiplying the complex conjugate of \eqref{Model2a} by $\Psi u$, taking the imaginary parts, and integrating over $\mathbb{R}\times (t_0,t)$, we obtain

$$\int_0^\infty \Psi |u|^2 dx = \int_0^\infty \Psi|u_0|^2dx -2\int_{t_0}^t\Im \int_0^\infty \varphi u\bar{u}_xdxds.$$
\end{proof}

\subsection{Estimates}
The corresponding  estimate for the half-line problem is below.

\begin{lem}\label{boundaryest} Let $u\in C(t_0,t_0+T_0; H^1)$ be a solution of \eqref{Model2a}-\eqref{Model2c}. Then, for $\displaystyle \delta=\frac{1}{2}\|\varphi'''\|_\infty\|u_0\|_2^2$, one has
\begin{equation}\label{varphi2lemB}
  -\Im \int_0^\infty \varphi u\bar{u}_x dx\le - \Im \int_0^\infty \varphi u_0\bar{u}_0' dx + (2E(u_0)+\delta)(t-t_0)
\end{equation} for $t\in [t_0,t_0+T_0)$.
\end{lem}

\begin{proof} First of all we observe that $|\varphi'(x)|\le 1$ and $|\varphi'''(x)|\le 2C_1$ where $C_1=\frac{1}{2}\|\varphi'''\|_\infty$ is a positive constant.  Using this observation, the definition of energy and the conservation of density in \eqref{varphi01B}, we obtain

\begin{multline}-\Im \int_0^\infty \varphi u\bar{u}_x dx + \Im \int_0^\infty \varphi u_0\bar{u}_0' dx\\
\le \int_{t_0}^t \left(2E(u_0)+\frac{2-r}{2+r}a(2\Omega s)|u(s,0)|^{r+2}+C_1\|u_0\|_{L^2(\mathbb{R}_+)}^2\right)ds.\end{multline}

Using the assumptions $r\ge 2$ and $a(2\Omega t)>0$ in $[t_0,t_0+T)$, we get the lemma.
\end{proof}
\subsection{Proof of collapse for the half-line problem}
The rest of the proof of the blow-up result can be carried out exactly as in Section \ref{Theorem1Details} by analysing the behaviour of a quadratic polynomilal. The details are omitted.\qed

\section{A Numerical Example}\label{NumSec}  The conditions in Theorem \ref{mainresult-wl} might seem very intricate and one might wonder whether there is any initial datum $u_0$ and $A_\Omega$ (with $\Omega$ big - the interesting case) such that the given assumptions are satisfied.  One can show that there are indeed suitable initial data for the blow-up to occur.  Let us prove this fact for example by finding an initial datum $u_0$ for problem \eqref{Model1a}-\eqref{Model1b} satisfying the assumptions in Theorem \ref{mainresult-wl} (i) which has also infinite momentum, i.e., $xu_0\notin L^2(\mathbb{R})$.

Here, we are interested in proving collapse of solutions even under the effect of fast oscillations.  Therefore, we will consider a relatively large oscillation constant, take for instance $\Omega=100$ and define $A_\Omega(t)=c_0\cos^2(\Omega t)=c_0\cos^2(100t)$, where $c_0=73.55418773631645$.  The derivative of this function is $A_\Omega'(t)=-\Omega c_0\sin(2\Omega t)=-100c_0\sin(200 t).$  Hence, an appropriate pair of $t_0$ and $T$ is given by $\displaystyle t_0=\frac{3\pi}{400}$ and $\displaystyle T=\frac{\pi}{400}$ since $A_\Omega(t_0)> 0$ and $A_{\Omega}'(t)\ge 0$ for $t\in [t_0,t_0+T]$.  Note, that $T\approx 0.00785398$ is relatively much smaller than the case of slow oscillations, say when $\Omega=1$.  The graph of $A_\Omega$ is given in Figure \ref{aomega}.

\begin{figure}[h]
  \centering
   \includegraphics[scale=.75]{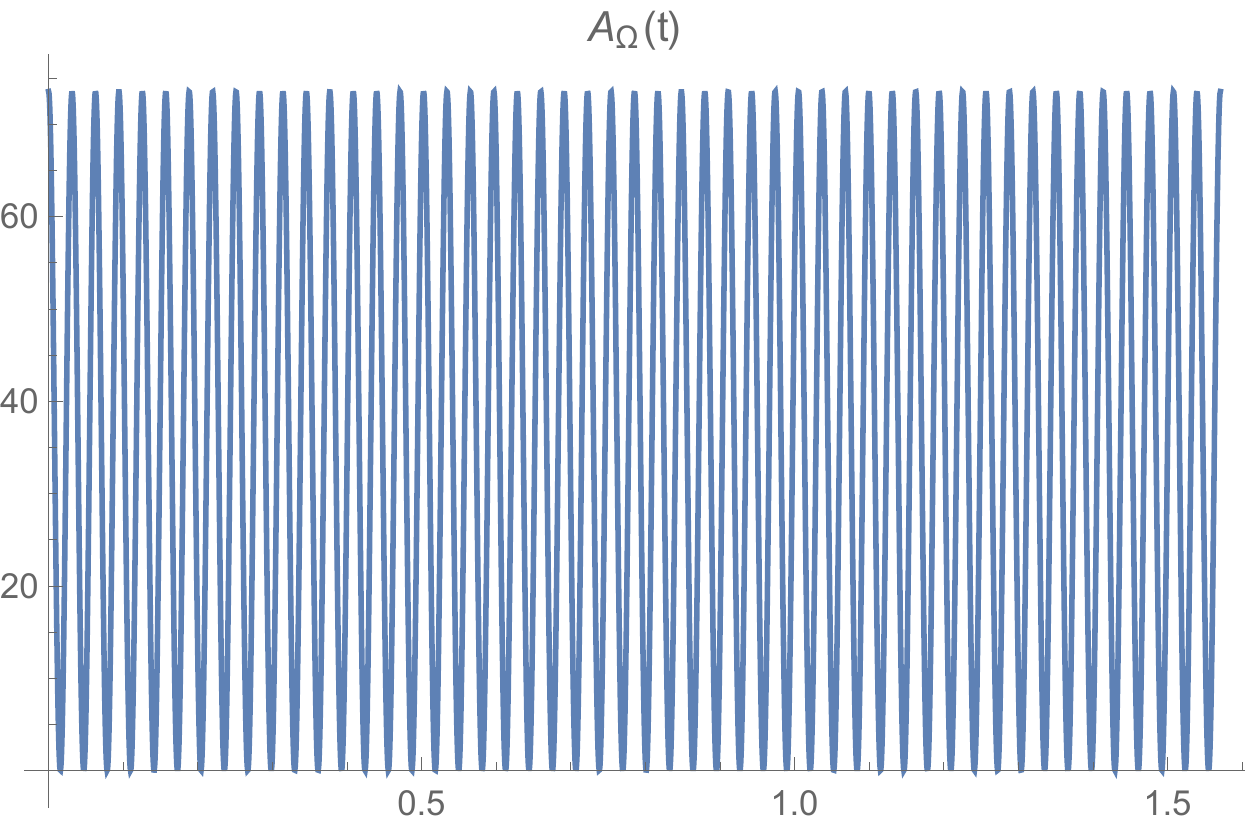}
  \caption{The oscillating coefficient $A_\Omega$ over an extended interval}\label{aomega}
\end{figure}

The first condition on $u_0$ is that it must be an element of $H^1(\mathbb{R})$.  That requires that $u_0$ and $\partial_x u_0$ must have sufficient decay as $|x|\rightarrow \infty.$  However, since we also want $u_0$ to have infinite momentum, the decay assumption on $u_0$ should not be too strong so that $xu_0\notin L^2(\mathbb{R})$ is satisfied.  This motivates us to consider a function $u_0$ which involves one or more terms similar to $\displaystyle\frac{1}{(x^2+1)^{\frac{1}{2}}}$.

We first construct a concrete example of a compactly supported weight function $\varphi$.  Note that the function $\varphi$ given in \eqref{varphidef} is assumed to stay smooth for $ 1.57735\approx 1+\frac{1}{\sqrt{3}}< |x|<2$ but an explicit definition on that portion was not given.  Since, our aim here is to give a numerical example, we first define this missing smooth portion of the function $\varphi$.  In order to do this, we use two mollifiers $l$ and $r$, where $l$ refers to the mollifier that we will apply to the the left hand side of the graph of $\varphi$ whereas $r$ refers to the mollifier that we will apply to the right hand side of the graph of $\varphi.$  To this end, we define
\begin{equation}\label{l-mol}
  l(x)=
     \begin{cases}
      Exp(1)Exp\left({\frac{-1}{(1 - (0.4(x + 1.6))^4)}}\right) &\quad\text{if } |x+1.6|<.4,\\
    0 &\quad\text{otherwise}
     \end{cases}
\end{equation} and

\begin{equation}\label{r-mol}
  r(x)=
     \begin{cases}
      Exp(1)Exp\left({\frac{-1}{(1 - (0.4(x - 1.6))^4)}}\right) &\quad\text{if } |x-1.6|<.4,\\
    0 &\quad\text{otherwise}
     \end{cases}
\end{equation}  The reason that we are using the fourth power in the term $(x \mp 1.6)^4$ is because when we apply these mollifiers to smoothly join two pieces of $\varphi$, we want $\varphi$ to have continuous derivatives at least up to order three. The graph of these mollifiers are given in Figure \ref{figlr}.

\begin{figure}[h]
  \centering
   \includegraphics[scale=.75]{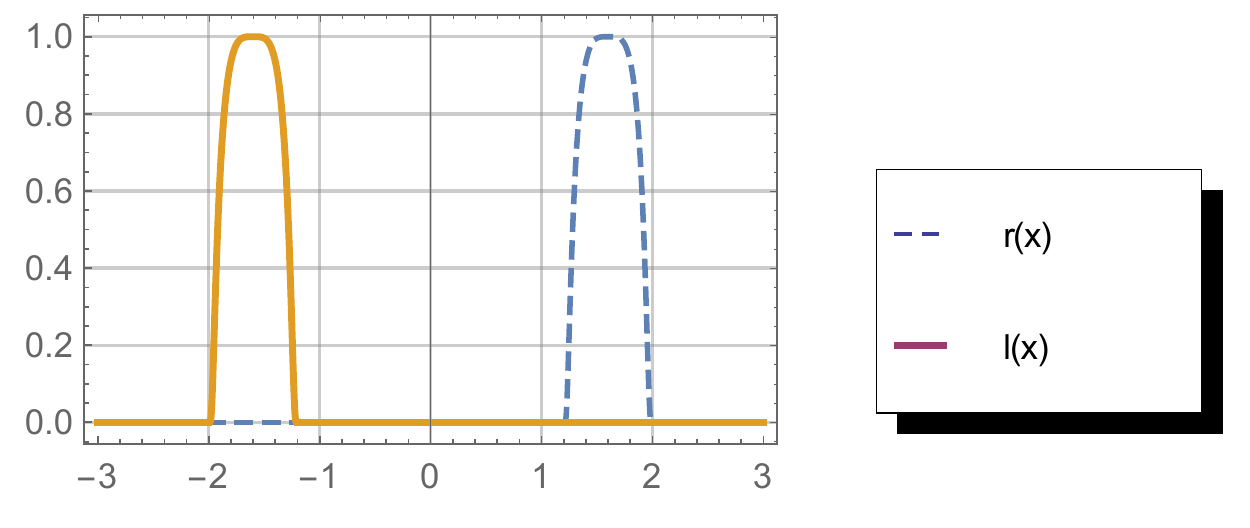}
  \caption{Mollifiers $l$ and $r$}\label{figlr}
\end{figure}

Now, we define (noticing that $1+\frac{1}{\sqrt{3}}\approx1.57735<1.6$):

\begin{equation}\label{varphidefAgain}
 \varphi(x) =
  \begin{cases}
      \hfill x    \hfill & \text{ if $|x|\le 1$}, \\
      \hfill x-(x-1)^3 \hfill & \text{ if $1<x \le  1+1/\sqrt{3}$}, \\
      \hfill -x+2(1+1/\sqrt{3})-(-x+1+2/\sqrt{3})^3 \hfill & \text{ if $1+1/\sqrt{3}<x \le 1.6 $}\\
      \hfill (x - (x - 1)^3)r(x)   \hfill & \text{ if $1.6< x< 2$}, \\
      \hfill x-(x+1)^3 \hfill & \text{ if $-(1+1/\sqrt{3})\le x<-1 $}, \\
      \hfill -x-2(1+1/\sqrt{3}) - (-x-1-2/\sqrt{3})^3   \hfill & \text{ if $-1.6\le  x< -(1+1/\sqrt{3})$}, \\
      \hfill (x - (x + 1)^3)l(x)   \hfill & \text{ if $-2<  x< -1.6$}, \\
      \hfill 0    \hfill & \text{ if $2\le |x|$}, \\
  \end{cases}
\end{equation}  This function is $C^3$ and satisfies the desired condition $\varphi'(x)\le 0$ for $|x|\ge 1+\frac{1}{\sqrt{3}}.$ Note that we have used vertical reflections at $x=\mp (1+1/\sqrt{3})$ to enforce the latter property. The graph of this compactly supported weight function, which satisfies all the conditions in \eqref{varphidef} is depicted in Figure \ref{phi}.

\begin{figure}[h]
  \centering
   \includegraphics[scale=.75]{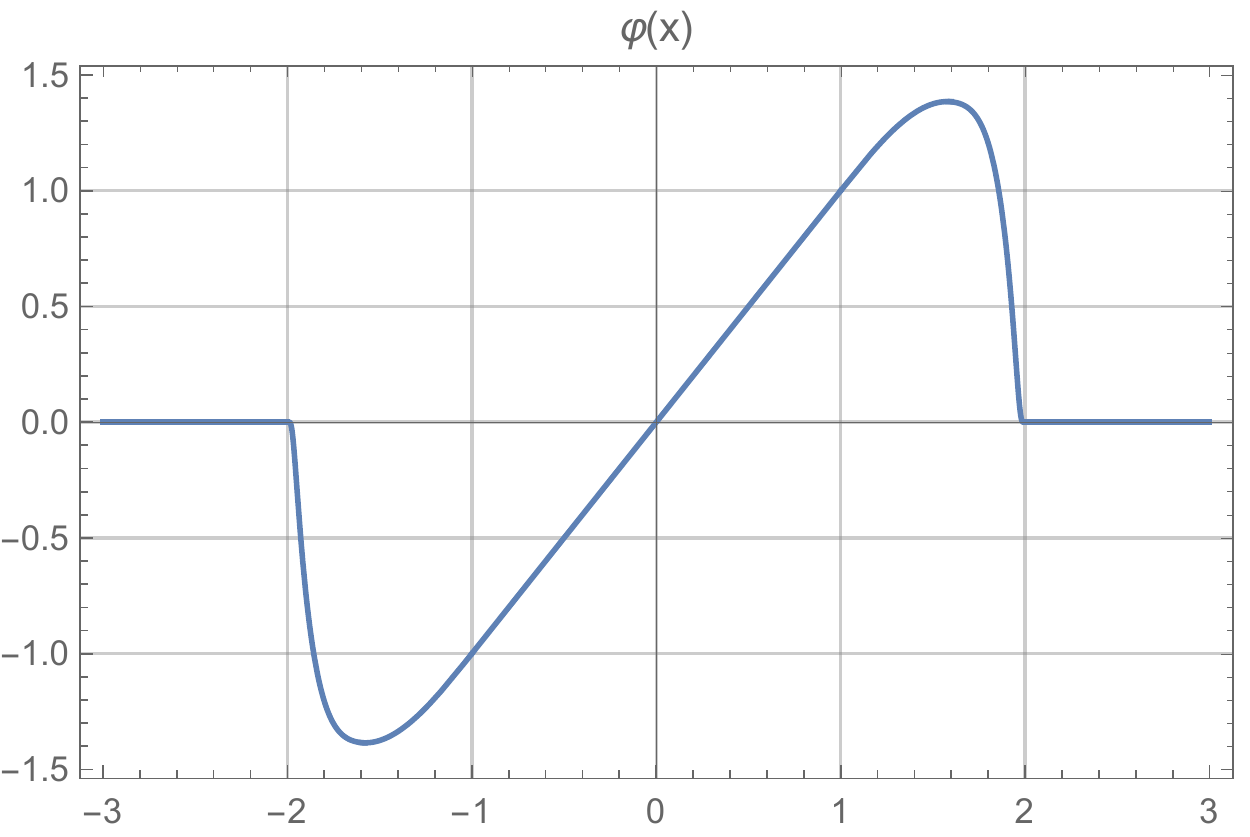}
  \caption{Weight function $\varphi$}\label{phi}
\end{figure}

One of the parameters in Theorem \ref{mainresult-wl} is $\delta$.  This constant needs to be specified before we proceed further.  In fact after a careful analysis of the proofs of Lemmas \ref{virial} and \ref{interiorest}, it follows that: $$\displaystyle \delta=\frac{1}{2}\sqrt{\frac{3}{8A_{\Omega}(t_0+T)}}\left(\|\varphi'''\|_\infty+{\max\left\{\sqrt{3}, \frac{1}{2}\|\varphi''\|_\infty\right\}^2}\right).$$

One immediate observation from the requirement $$\Im\int_{G} \varphi u_0\bar{u}_0' dx >0$$ is that the imaginary part of $u_0$ should not be zero and the real and imaginary parts should not be linearly dependent.  Therefore, $u_0$ should be in the form $u_0(x)=f(x)+ig(x)$ where $f$ and $g$ are linearly independent real valued functions on $\mathbb{R}$.

We will define $u_0$ to be zero at the left half line for simplicity so that we don't have to deal with the cancellation effects in some integrals which come from the oddness of $\varphi$. To this end, we first set a function $m=m(x)$ on $\mathbb{R}$ which will be tuned to be $u_0$ in a moment: \begin{equation}\label{defm}
                      m(x)=  \begin{cases}
                      0 & \quad\text{if } x<0,\\
                      \frac{x}{\sqrt{2}} & \quad\text{if } 0\le x<1,\\
                      \frac{1}{\sqrt{1+x^2}} & \quad\text{if } x\ge 1.
     \end{cases}
                    \end{equation} A graph of this function and its derivative are given in Figures \ref{mgraph} and \ref{dmgraph}
\begin{figure}[h]
  \centering
   \includegraphics[scale=.75]{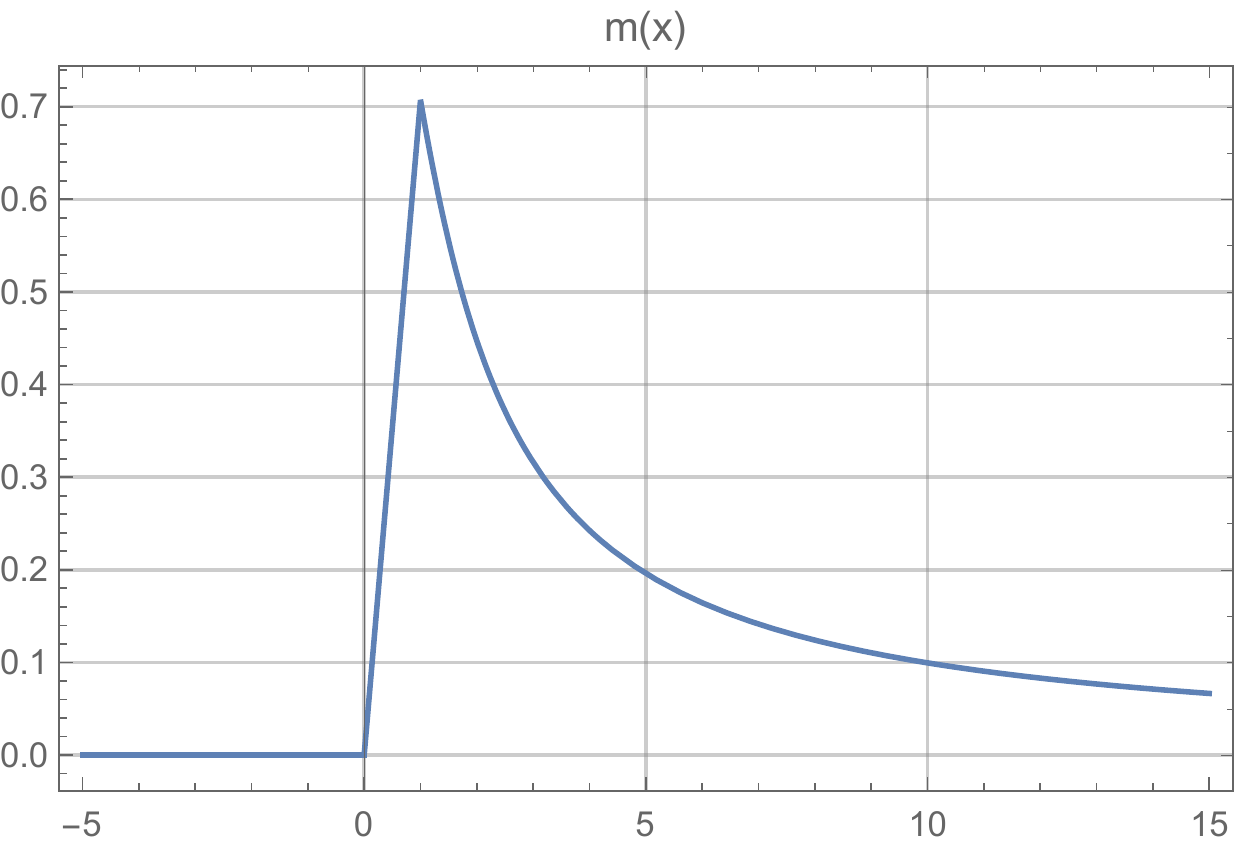}
  \caption{The graph of $m(x)$}\label{mgraph}
\end{figure}

\begin{figure}[h]
  \centering
   \includegraphics[scale=.75]{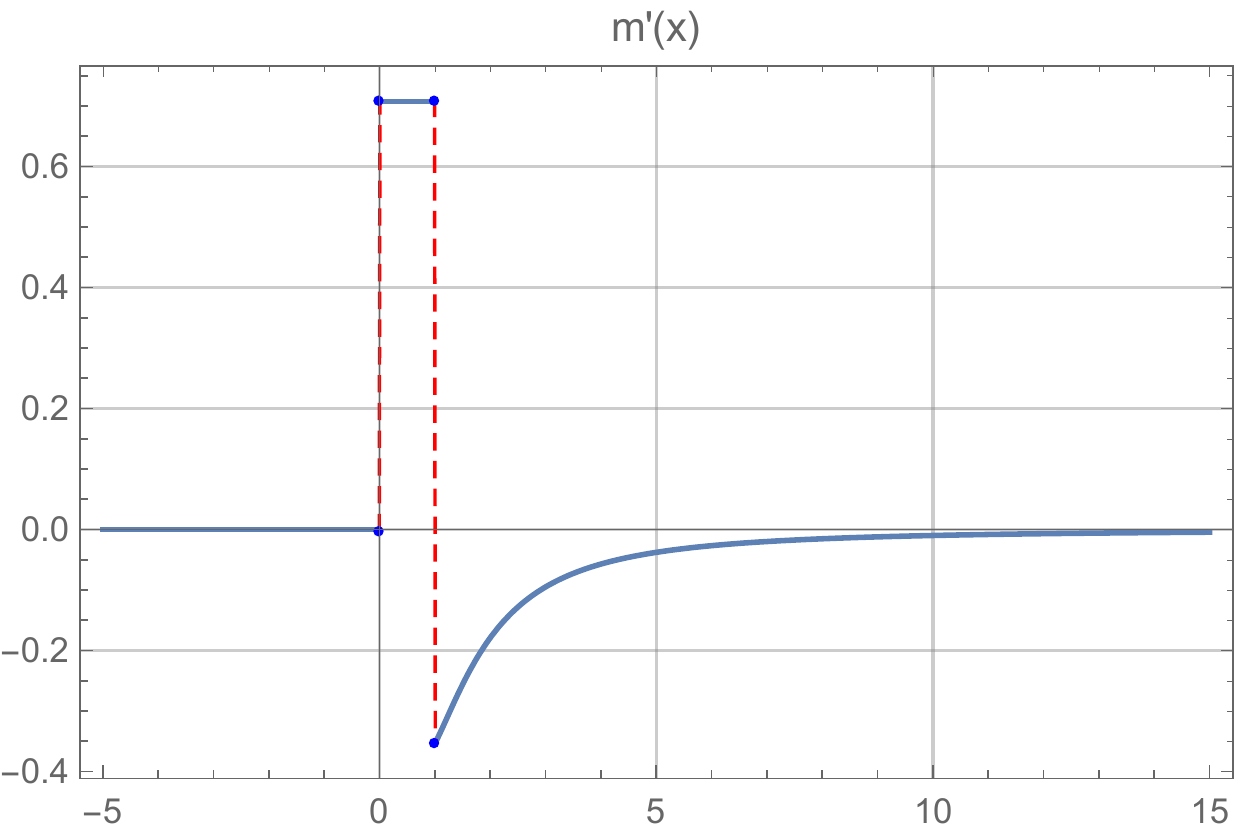}
  \caption{The graph of $m'(x)$}\label{dmgraph}
\end{figure}

It is easy to show that $m\in H^1(\mathbb{R})$ which is enough regularity for our purposes.  Now, we define $u_0$ as follows:
\begin{equation}\label{defu0}
                      u_0(x)=
                      \frac{1}{\sqrt{\rho}}\left[m\left(\frac{x}{\rho}\right) + im\left(\frac{2x}{\rho}\right)\right],
                    \end{equation} where $\rho>0$ is a fixed parameter to be specified in a moment.  One should carefully observe that as $\rho$ gets smaller $u_0$ gets concentrated around the origin more and $\|u_0\|_{L^2(|x|\ge 1)}$ gets smaller, see this in Figure \ref{u0graph}.  Intuitively, blow-up must occur if there is a nonlinear effect  strong enough to compete with dispersive effect.  Since, we are restricted to the case of positive energy solutions, in order to maximize the possiblity of collapse we consider a function for which both $\frac{A_\Omega(t_0)}{3}\|u_0\|_{L_6(\mathbb{R})}^6$ and $\|\partial_xu_0\|_{L^2(\mathbb{R})}^2$ are big. This is exactly what is achieved by the scaling introduced in \eqref{defu0}.  Therefore, we are motivated to take $\rho$ very small.
\begin{figure}[h]
  \centering
   \includegraphics[scale=.75]{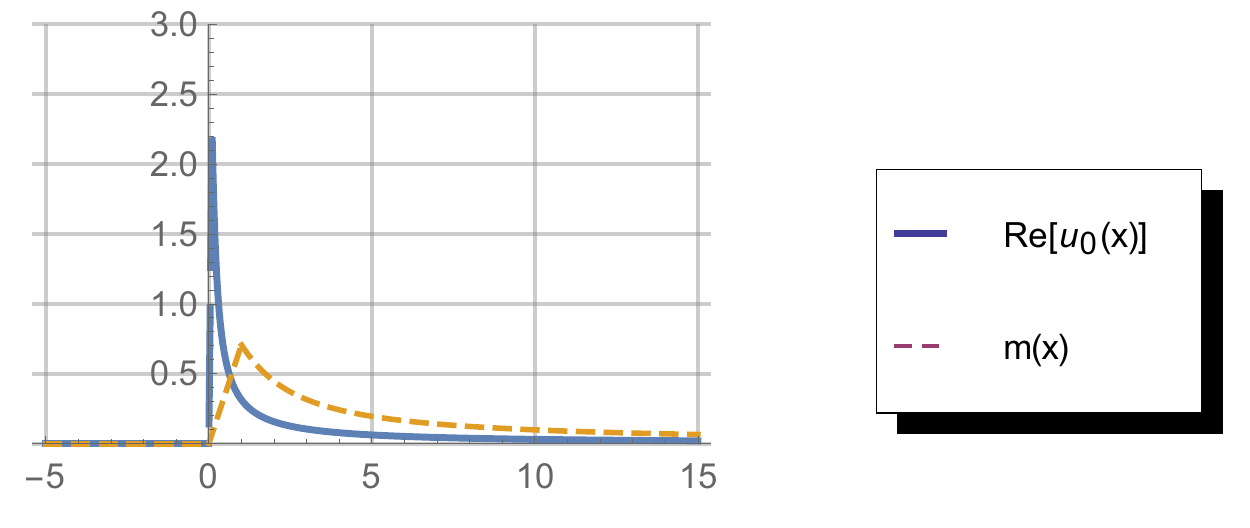}
  \caption{The graph of initial datum $\text{Re}[u_0(x)]$ ($\rho=0.1$) and $m(x)$}\label{u0graph}
\end{figure}

Now, we make the story more precise (with calculations performed in Wolfram Mathematica\textsuperscript{\textregistered}11).  We choose $\rho=10^{-10}$ (small for the reasons discussed above). Then, we have $\|u_0\|_{L^2(\mathbb{R})}^2=0.410875$, $\|xu_0\|_{L^2(\mathbb{R})}^2=\infty$ (as desired for $u_0$ to have infinite momentum), $\displaystyle \frac{A_\Omega(t_0)}{3}\|u\|_{L^6(\mathbb{R})}^6=1.64044\times 10^{20}$ (large as desired), $E(u_0)=32768>-\delta/2=-1.74158\times 10^6$ (note that energy is positive), $\Im\int_{G} \varphi u_0\bar{u}_0' dx =0.244626$ (strictly positive as desired and therefore $\lambda<0$),

$$\left(\Im\int_{G} \varphi u_0\bar{u}_0' dx\right)^2-(2E(u_0)+\delta)\left(\int_{G}\Psi|u_0|^2dx\right)=0.0590412>0$$ (strictly positive as desired and therefore $\lambda^2>4\alpha\beta$), where $$\left(\int_{G}\Psi|u_0|^2dx\right)=2.25667\times 10^{-10}.$$ Finally,
\begin{multline}\frac{\Im\int_{G} \varphi u_0\bar{u}_0' dx-\sqrt{\left(\Im\int_{G} \varphi u_0\bar{u}_0' dx\right)^2-(2E(u_0)+\delta)\int_{G}\Psi|u_0|^2dx}}{2E(u_0)+\delta}\\
=4.62803\times 10^{-10}<T=0.00785398,\end{multline} where $\delta=3.48315\times 10^6,$ and therefore $\theta_{-}<T$.  Finally, $$\displaystyle \alpha+\beta\theta_{-}^2=2.26427\times 10^{-10}<\frac{1}{2}\left(\frac{3}{8A_\Omega(t_0+T)}\right)^{\frac{1}{2}}=0.0357011.$$

Hence, $u_0$ satisfies all the conditions in Theorem \ref{mainresult-wl} (i) and therefore the corresponding solution blows-up after a short while before $t$ reaches $t_0+T$.

\section{Conclusion}

In this study, we proved that for some suitably chosen initial data the corresponding (negative or positive energy, finite or infinite momentum) solution of the Schrödinger equation, which is posed either on the real axis or on the half-line, with a focusing type nonlinear interior or boundary source term whose coefficient is time dependent (oscillating) will blow-up at the energy level.  Blow-up of solutions in the case of oscillations is very interesting since it is well-known that oscillations help to create a stabilizing effect on the solutions.  In the case of infinite momentum solutions, the blow-up phenomena was well-known when the coefficient of the nonlinear source term were a constant and energy were negative (see \cite{AASKD} and \cite{OT}).  Regarding oscillating nonlinearities, blow-up was shown only for the finite momentum solutions of the Schrödinger equation posed on the entire space.  It turns out that the finite momentum assumption is not necessary in the one dimensional setting, although it maybe in higher dimensions (see \cite{DO} and \cite{ZZ}).

\section*{Acknowledgements}
I would like to thank the anonymous referee for his careful reading of the manuscript and his several insightful comments and suggestions, which significantly contributed to improving the quality of this article.


\begin{thebibliography}{9}
\bibitem{abdul} F. K. Abdullaev and M. Salerno, Phys. Rev. A 72, 033617 (2005)
\bibitem{AASKD} A.S. Ackleh and K. Deng, \emph{On the critical exponent for the Schrödinger equation with a nonlinear boundary condition}, Differential Integral Equations  17  (2004),  no. 11--12, 1293--1307.
    \bibitem{alf2007} G.L. Alfimov, V.V. Konotop, P. Pacciani, Physical Review A 75(2) (2007) 023624
\bibitem{BO} A. Batal and T. Özsarı, \emph{Nonlinear Schrödinger equation on the half-line with nonlinear boundary condition}, Electron. J. Differential Equations 2016, Paper No. 222, 20 pp.
    \bibitem{BR} R. Balakrishan, \emph{Soliton propagation in nonuniform media}, Physical Review A 32 (2): 1144--1149
\bibitem{Cazenave} T. Cazenave, \emph{Semilinear Schrödinger equations}, Courant Lecture Notes in Mathematics, 10. New York University, Courant Institute of Mathematical Sciences, New York; American Mathematical Society, Providence, RI, 2003.
\bibitem{DO} I. Damergi and O. Goubet, \emph{Blow-up solutions to the nonlinear Schrödinger equation with oscillating nonlinearities.} J. Math. Anal. Appl.  352  (2009),  no. 1, 336–-344
\bibitem{Glassey} R. T. Glassey, On the blowing up of solutions to the Cauchy problem for nonlinear Schrödinger equations. J. Math. Phys. 18 (1977), no. 9, 1794--1797.
    \bibitem{AVG} A.V. Gurevich, \emph{Nonlinear Phonomena in the Ionosphere}, Berlin: Springer (1978)
    \bibitem{Holmer2015} J. Holmer, C. Liu, \emph{Blow-up for the 1D nonlinear Schrödinger equation with point nonlinearity I: Basic theory}, 	arXiv:1510.03491 [math.AP], preprint.
\bibitem{VKTO} V.K. Kalantarov and T. Özsarı, \emph{Qualitative properties of solutions for nonlinear Schrödinger equations with nonlinear boundary conditions on the half-line}, J. Math. Phys. 57 (2016), no. 2, 021511.
\bibitem{Kato87} T. Kato, \emph{On nonlinear Schrödinger equations}. Ann. Inst. H. Poincaré Phys. Théor. 46 (1987), no. 1, 113--129.
\bibitem{Kol2000} E. B. Kolomeisky, T. J. Newman, J. P. Straley, and X. Qi, Phys. Rev. Lett. 85, 1146 (2000)
\bibitem{Lieb} E.H. Lieb, R. Seringer, J. Yngvason, Phys. Rev. Lett. 91 (2003) 150401.
\bibitem{MB} B.A. Malomed, \emph{Nonlinear Schrödinger Equations}, in Scott Alwyn, Encyclopedia of Nonlinear Science, New York: Routledge, 639--643 (2005)
\bibitem{Molina} M. I. Molina and C. A. Bustamante, The attractive nonlinear delta-function potential,
arXiv:physics/0102053 [physics.ed-ph], preprint.
\bibitem{OT2} T. Ogawa and Y. Tsutsumi, \emph{Blow-up of $H^1$ solution for the nonlinear Schrödinger equation.} J. Differential Equations, 92 (1991), 317--330.
\bibitem{OT} T. Ogawa and Y. Tsutsumi, \emph{Blow-up of $H^1$ solutions for the one-dimensional nonlinear Schrödinger equation with critical power nonlinearity.} Proc. Amer. Math. Soc. 111 (1991), no. 2, 487--496.
    \bibitem{parades} B. Paredes, A. Videra, V. Murg, O. Mandel, S. Frölling, I. Cirac, G.V. Shlyapnikov,
T.W. Wänsch, I. Bloch, Nature (London) 249 (2004) 277.
    \bibitem{PS} L. Pitaevskii and S. Stringari, \emph{Bose-Einstein Condensation}, Oxford, U.K.: Clarendon (2003)
\bibitem{OStr} T. Özsarı, \emph{Well-posedness for nonlinear Schrödinger equations with boundary forces in low dimensions by Strichartz estimates}. J. Math. Anal. Appl. 424 (2015), no. 1, 487–508.
    \bibitem{Royden} H.L. Royden, P. Fitzpatrick, (2010). Real analysis. Boston: Prentice Hall.
    \bibitem{Sulem} C. Sulem and  P.L. Sulem, \emph{The Nonlinear {S}chrödinger equation:  Self-Focusing and Wave
Collapse, }Series  in  Mathematical  Sciences,  Volume  139,  Springer-Verlag,  xvi+350
pages, 1999.
\bibitem{yeh} P. Yeh, Optical Waves in Layered Media, New York: Wiley, 1988.
    \bibitem{ZS} V.E. Zakharov and A.B. Shabat, \emph{Exact theory of two dimensional self-focusing and one-dimensional self-modulation of waves in nonlinear media}, Soviet Physics JETP 34 (1972), no. 1, 62--69.
\bibitem{ZZ} J. Zhang and S. Zhu, \emph{Blow-up profile to solutions of NLS with oscillating nonlinearities.} NoDEA Nonlinear Differential Equations Appl.  19  (2012),  no. 2, 219–-234.


\end{thebibliography}
\end{document}